\documentclass[11pt,letterpaper]{amsart}

\usepackage{amssymb}
\usepackage{enumerate}

\newtheorem{theorem}{Theorem}[section]
\newtheorem{proposition}[theorem]{Proposition}
\newtheorem{corollary}[theorem]{Corollary}
\newtheorem{lemma}[theorem]{Lemma}

\newcommand{\Z}{\mathbb{Z}}
\newcommand{\F}{\mathbb{F}}
\newcommand{\C}{\mathbb{C}}
\newcommand{\R}{\mathbb{R}}
\newcommand{\N}{\mathbb{N}}
\newcommand{\Q}{\mathbb{Q}}

\newcommand{\Ac}{\mathcal{A}}
\newcommand{\Ec}{\mathcal{E}}

\newcommand{\e}{\epsilon}

\newcommand{\leg}[2]{\mbox{$({#1}\!\mid\!{#2})$}}
\newcommand{\norm}[1]{\lVert#1\rVert}

\newcommand{\abs}[1]{\lvert#1\rvert}
\newcommand{\bigabs}[1]{\big\lvert#1\big\rvert}

\newcommand{\Biggabs}[1]{\Bigg\lvert#1\Bigg\rvert}

\newcommand{\sums}[1]{\sum_{\substack{#1}}}

\newcommand{\stack}[2]{\genfrac{}{}{0pt}{}{#1}{#2}}

\def \bangle{ \atopwithdelims \langle \rangle}

\DeclareMathOperator{\E}{E}

\begin{document}

\title[$L^q$ norms of Fekete and related polynomials]{\boldmath $L^q$ norms of Fekete and related polynomials}

\author{Christian G\"unther}
\address{Department of Mathematics, Paderborn University, Warburger Str.\ 100, 33098 Paderborn, Germany.}
\email[Ch. G\"unther]{chriguen@math.upb.de}

\author{Kai-Uwe Schmidt}
\address{Department of Mathematics, Paderborn University, Warburger Str.\ 100, 33098 Paderborn, Germany.}
\email[K.-U. Schmidt]{kus@math.upb.de}

\thanks{The authors are supported by German Research Foundation (DFG)}

\date{02 February 2016}


\begin{abstract}
A Littlewood polynomial is a polynomial in $\C[z]$ having all of its coefficients in $\{-1,1\}$. There are various old unsolved problems, mostly due to Littlewood and Erd\H{o}s, that ask for Littlewood polynomials that provide a good approximation to a function that is constant on the complex unit circle, and in particular have small $L^q$ norm on the complex unit circle. We consider the Fekete polynomials
\[
f_p(z)=\sum_{j=1}^{p-1}\leg{j}{p}\,z^j,
\]
where $p$ is an odd prime and $\leg{\,\cdot\,}{p}$ is the Legendre symbol (so that $z^{-1}f_p(z)$ is a Littlewood polynomial). We give explicit and recursive formulas for the limit of the ratio of $L^q$ and $L^2$ norm of $f_p(z)$ when $q$ is an even positive integer and $p\to\infty$. To our knowledge, these are the first results that give these limiting values for specific sequences of nontrivial Littlewood polynomials and infinitely many $q$. Similar results are given for polynomials obtained by cyclically permuting the coefficients of Fekete polynomials and for Littlewood polynomials whose coefficients are obtained from additive characters of finite fields. These results vastly generalise earlier results on the $L^4$ norm of these polynomials.
\end{abstract}

\maketitle

\thispagestyle{empty}


\section{Introduction}

For real $\alpha\ge 1$, the $L^\alpha$ norm of a polynomial $f(z)$ in $\C[z]$ on the complex unit circle is
\[
\norm{f}_\alpha=\bigg(\frac{1}{2\pi}\int_0^{2\pi}\;\abs{f(e^{ i\theta})}^\alpha\;d\theta\bigg)^{1/\alpha},
\]
and its supremum norm is $\norm{f}_\infty=\max_{\theta\in[0,2\pi]}\abs{f(e^{i\theta})}$. There are various extremal problems, originally raised by Erd{\H{o}s}, Littlewood, and others, concerning the behaviour of such norms for polynomials with all coefficients in $\{-1,1\}$, which are today called \emph{Littlewood polynomials} (see Littlewood~\cite{Lit1968}, Borwein~\cite{Bor2002}, and Erd{\'e}lyi~\cite{Erd2002} for surveys on selected problems). Roughly speaking, such problems ask for Littlewood polynomials $f(z)$ that provide a good approximation to a function that is constant on the unit circle. Note that this constant is necessarily $\norm{f}_2=\sqrt{1+\deg f}$. 
\par
Several conjectures have been posed that address the question of what is the best approximation in a certain sense. For example, Golay~\cite{Gol1982} conjectured that there exists a constant $c$ such that $\norm{f}_4/\norm{f}_2\ge 1+c$ for every nonconstant Littlewood polynomial~$f$ and Littlewood~\cite{Lit1966} conjectured that there is no such constant. Golay's conjecture implies another famous conjecture due to Erd{\H{o}s}~\cite{Erd1995},~\cite{NewByr1990}, which states that there exists a constant~$c'$ such that $\norm{f}_\infty/\norm{f}_2\ge 1+c'$ for every nonconstant Littlewood polynomial~$f$. All these conjectures are wide open.
\par
Borwein and Lockhart~\cite{BorLoc2001} proved that, if $f_n$ is a random polynomial of degree $n-1$, then
\[
\lim_{n\to\infty}\E\left(\frac{\norm{f_n}_\alpha}{\sqrt{n}}\right)^\alpha=\Gamma(1+\alpha/2)
\]
and $(\norm{f_n}_\alpha/\sqrt{n})^\alpha$ is asymptotically concentrated around its expectation (see also Choi and Erd{\'e}lyi~\cite{ChoErd2014} for more results on $L^\alpha$ norms of random Littlewood polynomials). Littlewood~\cite{Lit1968} (and independently Newman and Byrnes~\cite{NewByr1990} and H{\o}holdt, Jensen, and Justesen~\cite{HohJenJus1985}) determined the $L^4$ norm of the Rudin-Shapiro polynomials~\cite{Sha1951},~\cite{Rud1959}. More generally, a conjecture attributed in~\cite{DocHab2004} to Saffari asserts that, if $q$ is a positive integer and $f_n$ is a Rudin-Shapiro polynomial of degree $n-1$, then
\[
\lim_{n\to\infty}\left(\frac{\norm{f_n}_{2q}}{\sqrt{n}}\right)^{2q}=\frac{2^q}{q+1}.
\]
This conjecture is true for $q\le 27$ by combining results of Doche and Habsieger~\cite{DocHab2004} and Taghavi and Azadi~\cite{TagAza2008}, but the general problem remains~open.
\par
In this paper we consider the following families of polynomials. For an odd prime~$p$, the \emph{Fekete polynomial} of degree $p-1$ is
\[
f_p(z)=\sum_{j=1}^{p-1}\leg{j}{p}\,z^j,
\]
where $\leg{\,\cdot\,}{p}$ is the Legendre symbol. Note that $z^{-1}f_p(z)$ is a Littlewood polynomial, which has the same $L^\alpha$ norm as $f_p(z)$. For a Mersenne number $n=2^k-1$, a \emph{Galois polynomial} of degree $n-1$ is the Littlewood polynomial
\[
g_n(z)=\sum_{j=0}^{n-1}\psi(\theta^j)\,z^j,
\]
where $\theta$ is a primitive element of $\F_{2^k}$ and $\psi$ is a nontrivial additive character of $\F_{2^k}$. Fekete polynomials appear frequently in the context of extremal polynomial problems~\cite{Mon1980}, \cite{HohJen1988}, \cite{JenJenHoh1991}, \cite{ConGraPooSou2000}, \cite{BorChoYaz2001}, \cite{BorCho2002}, \cite{JedKatSch2013a}, \cite{JedKatSch2013b}, \cite{Kat2013} and have been studied extensively now for over a century~\cite{FekPol1912}.
\par
Erd{\'e}lyi~\cite{Erd2012} established the order of growth of the $L^\alpha$ norm of Fekete polynomials. H{\o}holdt and Jensen~\cite{HohJen1988} proved that, for Fekete polynomials~$f_p(z)$,
\begin{align*}
\lim_{p\to\infty}\left(\frac{\norm{f_p}_4}{\sqrt{p}}\right)^4&=\frac{5}{3}.
\intertext{In fact Borwein and Choi~\cite{BorCho2002} established exact expressions for $\norm{f_p}_4$ in terms of the class number of $\Q(\sqrt{-p})$. Jensen, Jensen, and H{\o}holdt~\cite{JenJenHoh1991} proved that, for Galois polynomials $g_n(z)$,}
\lim_{n\to\infty}\left(\frac{\norm{g_n}_4}{\sqrt{n}}\right)^4&=\frac{4}{3}.
\end{align*}
These are in fact special cases of our main results (see Theorems~\ref{thm:norm_fekete} and~\ref{thm:norm_galois}), which provide corresponding limiting values for the $L^{2q}$ norms of Fekete and Galois polynomials for all positive integers $q$. To our knowledge, these are the first results that give these limiting values for specific sequences of nontrivial Littlewood polynomials and infinitely many $q$.
\par
We also consider the \emph{shifted} Fekete polynomials
\[
f^r_p(z)=\sum_{j=0}^{p-1}\leg{j+r}{p}\,z^j,
\]
where $r$ is an integer, which can depend on $p$. It is known~\cite{HohJen1988} that, if $r/p\to R$ as $p\to\infty$, then
\begin{equation}
\lim_{p\to\infty}\left(\frac{\norm{f^r_p}_4}{\sqrt{p}}\right)^4=\frac{7}{6}+\frac{1}{2}(4\abs{R}-1)^2\quad\text{for $\abs{R}\le \frac{1}{2}$}.   \label{eqn:limit_fekete_q2}
\end{equation}
Again, this is a special case of a more general result (see Theorem~\ref{thm:norm_shifted_fekete}). Note that a shifted Fekete polynomial is not necessarily a Littlewood polynomial since one of its first $p$ coefficients is zero. However changing this coefficient to $-1$ or $1$ does not affect the asymptotic behaviour of the $L^\alpha$ norm.


\section{Results}

We begin with establishing some notation that is required to state our results. For a positive integer $m$, let $\Pi_m$ be the set of partitions of $\{1,2,\dots,m\}$. For $\pi\in\Pi_m$, we refer to the elements of $\pi$ as \emph{blocks} and we say that $\pi$ is \emph{even} if each block of $\pi$ has even cardinality.
\par
For a positive integer $n$ and real $x$, we define the \emph{generalised Eulerian numbers} to be
\begin{equation}
{n\bangle x}=\sum_{j=0}^{\lfloor x+1\rfloor}(-1)^j{n+1\choose j}(x+1-j)^n.  \label{eqn:def:Eulerian_numbers}
\end{equation}
Note that ${n\bangle x}$ is nonzero only for $x\in(-1,n)$. If $x$ is integral, then ${n\bangle x}$ is an Eulerian number in the usual sense. We refer to the book~\cite{Pet2015} for the combinatorial significance of Eulerian numbers and to~\cite{WanXuXu2010} for a natural interpretation of generalised Eulerian numbers in terms of splines.
\par
The \emph{signed tangent numbers} $T(k)$ are defined by
\begin{equation}
\log\cosh(z)=\sum_{k=1}^\infty\frac{T(k)}{(2k)!}\,z^{2k}\quad\text{for $\abs{z}<\pi/2$}.   \label{eqn:log_cosh}
\end{equation}
They are scaled versions of Bernoulli numbers and $\abs{T(k)}=(-1)^{k+1}T(k)$ are known as the \emph{tangent} or \emph{zag} numbers, which appear in~\cite{Slo} as $\text{\ttfamily A000182}=[1,2,16,272,7936,353792,\dots]$. The numbers $T(k)$ can be recursively determined via
\[
T(k)=1-\sum_{j=1}^{k-1} {2k-1\choose 2j-1}T(j)\quad\text{for $k\ge 1$},
\]
which can be deduced from Lemma~\ref{lem:exp_formula}.
\par
For Fekete polynomials we have the following result.
\begin{theorem}
\label{thm:norm_fekete}
Let $q$ be a positive integer and let $f_p(z)$ be the Fekete polynomial of degree $p-1$. Then
\[
\lim_{p\to\infty}\left(\frac{\norm{f_p}_{2q}}{\sqrt{p}}\right)^{2q}=\sums{\pi\in\Pi_{2q}\\[0.5ex]\text{$\pi$ even}}\;\sums{a_1,\dots,a_\ell\in\Z\\a_1+\cdots+a_\ell=q}\;\prod_{i=1}^\ell\;\frac{T(N_i)}{(2N_i-1)!}{2N_i-1\bangle a_i-1},
\]
where $\pi=\{B_1,\dots,B_\ell\}$ and $N_i=\abs{B_i}/2$ for all $i$.
\end{theorem}
\par
The following corollary provides an efficient way to compute the limiting values in Theorem~\ref{thm:norm_fekete}.
\begin{corollary}
\label{cor:norm_fekete}
Set $F(0,0)=1$ and, for $1\le m\le 2k-1$, define the numbers $F(k,m)$ recursively by
\[
F(k,m)=\sum_{j=1}^k\binom{2k-1}{2j-1}\frac{T(j)}{(2j-1)!}\,\sum_i{2j-1\bangle i-1}F(k-j,m-i),
\]
where the inner sum is over all $i$ such that $F(k-j,m-i)$ is defined. Let~$q$~be a positive integer and let $f_p(z)$ be the Fekete polynomial of degree $p-1$. Then
\[
\lim_{p\to\infty}\left(\frac{\norm{f_p}_{2q}}{\sqrt{p}}\right)^{2q}=F(q,q).
\]
\end{corollary}
\par
For $k\ge 1$, the numbers $(2k-1)!\,F(k,m)$ identified in Corollary~\ref{cor:norm_fekete} define a triangular array of integers, whose first four rows are given by:
\[
\setlength{\arraycolsep}{9pt}
\begin{array}{ccccccccc}
 & & & & 1 & & & &\\
 & & & -2 & 10 & -2 & & &\\
 & & 16 & -184 & 456 & -184 & 16 &\\
 & -272 & 5776 & -30736 & 55504 & -30736 & 5776 & -272 &
\end{array}
\]
The first and last entry in row $k$ equals $T(k)$ and the central entry in row~$k$ divided by $(2k-1)!$ equals the limiting value in Corollary~\ref{cor:norm_fekete} for $k=q$. The first eight of these limiting values are:
\vspace*{1.5ex}
\[
1,\,\frac{5}{3},\,\frac{19}{5},\,\frac{3469}{315},\,\frac{21565}{567},\,\frac{7760593}{51975},\,\frac{12478099}{19305},\,\frac{643983856759}{212837625}.
\]
\vspace*{-0.5ex}
\par
We now turn to Galois polynomials. Let $J_0(z)$ be the zeroth Bessel function of the first kind and define the numbers $C(k)$ via
\begin{equation}
\log(J_0(2\sqrt{z}))=\sum_{k=1}^\infty \frac{(-1)^{k}\,C(k)}{(k!)^2}\,z^k.   \label{eqn:def_carlitz_numbers}
\end{equation}
We call these numbers the \emph{signed Carlitz numbers}. The corresponding unsigned numbers $\abs{C(k)}=(-1)^{k+1}C(k)$ have been extensively studied by Carlitz~\cite{Car1963} and appear in \cite{Slo} as $\text{\ttfamily A002190}=[0,1,1,4,33,456,9460,\dots]$ (which starts at $k=0$ with $C(0)=0$). The numbers $C(k)$ can be recursively determined via
\[
C(k)=1-\sum_{j=1}^{k-1}{k\choose j}{k-1\choose j-1}C(j)\quad\text{for $k\ge 1$},
\]
which again can be deduced from Lemma~\ref{lem:exp_formula}.
\par
For Galois polynomials we have the following result.
\begin{theorem}
\label{thm:norm_galois}
Let $q$ be a positive integer and let $g_n(z)$ be a Galois polynomial of degree $n-1$. Then
\[
\lim_{n\to\infty}\left(\frac{\norm{g_n}_{2q}}{\sqrt{n}}\right)^{2q}\!\!=\sum_{\pi\in\Pi_q}\!{q\choose N_1,\dots,N_\ell}\!\sums{a_1,\dots,a_\ell\in\Z\\a_1+\cdots+a_\ell=q}\;\prod_{i=1}^\ell\;\frac{C(N_i)}{(2N_i-1)!}{2N_i-1\bangle a_i-1},
\]
where $\pi=\{B_1,\dots,B_\ell\}$ and $N_i=\abs{B_i}$ for all $i$.
\end{theorem}
\par
We have the following counterpart of Corollary~\ref{cor:norm_fekete} for Galois polynomials.
\begin{corollary}
\label{cor:norm_galois}
Set $G(0,0)=1$ and, for $1\le m\le 2k-1$, define the numbers $G(k,m)$ recursively by
\[
G(k,m)=\sum_{j=1}^k\binom{k}{j}\binom{k-1}{j-1}\frac{C(j)}{(2j-1)!}\,\sum_i{2j-1\bangle i-1}G(k-j,m-i),
\]
where the inner sum is over all $i$ such that $G(k-j,m-i)$ is defined. Let~$q$~be a positive integer and let $g_n(z)$ be a Galois polynomial of degree~$n-1$. Then
\[
\lim_{n\to\infty}\left(\frac{\norm{g_n}_{2q}}{\sqrt{n}}\right)^{2q}=G(q,q).
\]
\end{corollary}
\par
For $k\ge 1$, the numbers $(2k-1)!\, G(k,m)$ identified in Corollary~\ref{cor:norm_galois} also define a triangular array of integers, whose first four rows are given by:
\[
\setlength{\arraycolsep}{9pt}
\begin{array}{ccccccccc}
 & & & & 1 & & & &\\
 & & & -1 & 8 & -1 & & &\\
 & & 4 & -76 & 264 & -76 & 4 &\\
 & -33 & 1248 & -9735 & 22080 & -9735 & 1248 & -33 &\\
\end{array}
\]
The first and last entry in row $k$ equals $C(k)$ and the central entry in row~$k$ divided by $(2k-1)!$ equals the limiting value in Corollary~\ref{cor:norm_galois} for $k=q$. The first eight of these limiting values are:
\vspace*{1.5ex}
\[
1,\,\frac{4}{3},\,\frac{11}{5},\,\frac{92}{21},\,\frac{15481}{1512},\,\frac{411913}{15120},\,\frac{2482927}{30888},\,\frac{4181926481}{16216200}.
\]
\vspace*{-0.5ex}
\par
In what follows we consider the shifted Fekete polynomials.
\begin{theorem}
\label{thm:norm_shifted_fekete}
Let $q$ be a positive integer and let $f_p^r(z)$ be a shifted Fekete polynomial corresponding to the Fekete polynomial of degree $p-1$. If $r/p\to R$ as $p\to\infty$, then
\[
\lim_{p\to\infty}\left(\frac{\norm{f_p^r}_{2q}}{\sqrt{p}}\right)^{2q}\!=\!\sums{\pi\in\Pi_{2q}\\[0.5ex]\text{$\pi$ even}}\;\sums{a_1,\dots,a_\ell\in\Z\\a_1+\cdots+a_\ell=q}\prod_{i=1}^\ell\;\frac{T(N_i)}{(2N_i-1)!}{2N_i-1\bangle 2R(N_i-P_i)+a_i-1},
\]
where $\pi=\{B_1,\dots,B_\ell\}$, $N_i=\abs{B_i}/2$, and $P_i=\abs{\{x\in B_i:x>q\}}$ for all~$i$.
\end{theorem}
\par
Note that, for $R=0$, Theorem~\ref{thm:norm_shifted_fekete} reduces to Theorem~\ref{thm:norm_fekete}. We are not aware of a computationally efficient version of Theorem~\ref{thm:norm_shifted_fekete} in a spirit similar to Corollaries~\ref{cor:norm_fekete} and~\ref{cor:norm_galois}.
\par
It follows from Theorem~\ref{thm:norm_shifted_fekete} that, for each positive integer $q$, there exists a function $\varphi_q:\R\to\R$ such that, if $r/p\to R$, then
\[
\lim_{p\to\infty}\left(\frac{\norm{f_p^r}_{2q}}{\sqrt{p}}\right)^{2q}=\varphi_q(R).
\]
Since the generalised Eulerian numbers ${n\bangle x}$ are continuous piecewise polynomial functions of $x$, the functions $\varphi_q$ are also continuous piecewise polynomial functions. It follows from Theorem~\ref{thm:norm_shifted_fekete} that $\varphi_q(x+1/2)=\varphi_q(x)$ for all $x\in\R$. It can also be shown that $\varphi_q(-x)=\varphi_q(x)$ for all $x\in\R$, so that it is sufficient to know $\varphi_q(x)$ for $x\in[0,1/2)$. We have for example
\begin{align*}
\varphi_2(x)&=\frac{7}{6}+\frac{1}{2}(4x-1)^2\quad\text{for $0\le x\le \frac{1}{2}$},
\intertext{in accordance with~\eqref{eqn:limit_fekete_q2},}
\varphi_3(x)&=\frac{31}{20}+\frac{3}{4}(4x-1)^2(16x^2-8x+3)\quad\text{for $0\le x\le \frac{1}{2}$},\\
\intertext{and}
\varphi_4(x)&=\begin{cases}
\phi(x)     & \text{for $0\le x\le 1/4$}\\[.5ex]
\phi(1/2-x) & \text{for $1/4\le x\le 1/2$},
\end{cases}
\end{align*}
where
\[
\phi(x)=\frac{653}{280}+\frac{1}{72} \, {\left(4 \, x - 1\right)}^{2} {\left(60416 \, x^{4} - 52736 \, x^{3} + 20208 \, x^{2} - 4216 \, x + 625\right)}.
\]
For $q\in\{2,3,4\}$, it is readily verified that the function $\varphi_q$ attains its global minimum at a unique point in $[0,1/2)$, namely at $1/4$. We could not prove that this is true for all $q>1$, but conjecture that this is the case. For convenience, we provide the first eight values of $\varphi_q(1/4)$ (starting with $q=1$):
\vspace*{1.5ex}
\[
1,\,\frac{7}{6},\,\frac{31}{20},\,\frac{653}{280},\,\frac{
71735}{18144},\,\frac{24880549}{3326400},\,\frac{72207143}{4633200},\,
\frac{960901090937}{27243216000}.
\]
\vspace*{-0.5ex}
\par
We shall prove our results for Fekete and Galois polynomials in Sections~\ref{sec:fekete} and~\ref{sec:galois}, respectively.
\par
We note that it is also possible to define shifted Galois polynomials by cyclically permuting the coefficients of a Galois polynomial. However every such polynomial is again a Galois polynomial. It should also be noted that our methods can be used to establish similar results for polynomials obtained by periodically appending or truncating monomials in Fekete or Galois polynomials, as considered in~\cite{JedKatSch2013a} and~\cite{JedKatSch2013b}.


\section{Calculation of $L^{2q}$ norms}

We begin with establishing some notation that will be used throughout this paper. For a positive integer $n$, we write $e_n(x)=\exp(2\pi ix/n)$. Let $f(z)=\sum_{j=0}^{n-1}a_jz^j$ be a polynomial of degree $n-1$ in $\C[z]$ and let~$r$ be an integer. Define the \emph{shifted} polynomial
\[
f^r(z)=\sum_{j=0}^{n-1}a_{j+r}z^j,
\]
where we extend the definition of $a_j$ so that $a_{j+n}=a_j$ for all $j\in\Z$. We shall express the $L^{2q}$ norm of this polynomial in a form that will be convenient for us later.
\par
To do so, we associate with $f$ the function $L_f:(\Z/n\Z)^{2q}\to\C$ given~by
\[
L_f(t_1,\dots,t_{2q})=\frac{1}{n^{q+1}}\;\sum_{m\in\Z/n\Z}\;\prod_{k=1}^q\;f(e_n(m+t_k))\overline{f(e_n(m+t_{q+k}))}
\]
and define another function $h_{n,r}:(\Z/n\Z)^{2q}\to\C$ by
\[
h_{n,r}(t_1,\dots,t_{2q})=\sums{0\le j_1,\dots,j_{2q}<n\\j_1+\dots+j_q=j_{q+1}+\dots+j_{2q}}\prod_{k=1}^q\overline{e_n(t_k(j_k+r))}e_n(t_{q+k}(j_{q+k}+r)).
\]
The following proposition will be the starting point to prove our main results.
\begin{proposition}
\label{pro:norm_from_L_and_h}
Let $q$ be a positive integer, let $f(z)$ be a polynomial in~$\C[z]$ of degree $n-1$, and let $r$ be an integer. Then
\[
\norm{f^r}_{2q}^{2q}=\frac{1}{n^{q}}\sum_{t\in(\Z/n\Z)^{2q}}L_f(t)\,h_{n,r}(t).
\]
\end{proposition}
\begin{proof}
Write $f(z)=\sum_{j=0}^{n-1}a_jz^j$. From
\[
\norm{f^r}_{2q}^{2q}=\frac{1}{2\pi}\int_0^{2\pi}\Big[f^r(e^{i\theta})\overline{f^r(e^{i\theta})}\Big]^q\,d\theta
\]
we obtain
\[
\norm{f^r}_{2q}^{2q}=\sums{0\le j_1,\dots,j_{2q}<n\\j_1+\cdots+j_q=j_{q+1}+\cdots+j_{2q}}\prod_{k=1}^qa_{j_k+r}\;\overline{a_{j_{q+k}+r}}.
\]
Now it is readily verified that
\[
a_j=\frac{1}{n}\sum_{s\in\Z/n\Z}f(e_n(s))\,e_n(-sj),
\]
giving
\[
\norm{f^r}_{2q}^{2q}=\frac{1}{n^{2q}}\sum_{s_1,\dots,s_{2q}\in\Z/n\Z}\;h_{n,r}(s_1,\dots,s_{2q})\;\prod_{k=1}^q\;f(e_n(s_k))\;\overline{f(e_n(s_{q+k}))}.
\]
Re-index the summation with $s_i=m+t_i$ for all $i$ and then sum over $m\in\Z/n\Z$ to obtain the statement in the proposition.
\end{proof}
\par
We also need the following estimate.
\begin{lemma}
\label{lem:error_term}
There exists a constant $C_q$, depending only on $q$, such that
\[
\sum_{t\in(\Z/n\Z)^{2q}}\abs{h_{n,r}(t)}\le C_q\,n^{2q}(\log n)^{2q-1}
\]
for all $r$.
\end{lemma}
\begin{proof}
After re-indexing the summation in the definition of $h_{n,r}(t)$, the statement of the lemma is equivalent to
\begin{equation}
\sum_{t_1,\dots,t_{2q}\in\Z/n\Z}\Bigg|\sums{0\le j_1,\dots,j_{2q}<n\\j_1+\cdots+j_{2q}=q(n-1)}e_n(t_1j_1+\cdots+t_{2q}j_{2q})\Bigg|\le C_q\,n^{2q}(\log n)^{2q-1}.   \label{eqn:lemma_error_term_to_show}
\end{equation}
For a positive integer $d$ let $P\subseteq[0,1]^d$ be a polyhedron and let 
\[
F_n(z_1,\dots,z_d)=\sum_{(j_1,\dots,j_d)\in\Z^d\cap (n-1)P}z_1^{j_1}\cdots z_d^{j_d}
\]
be a polynomial in $\C[z_1,\dots,z_d]$. Write
\[
S_n=\sum_{s_1,\dots,s_d\in\Z/n\Z}\bigabs{F_n(e^{2\pi is_1/n},\dots,e^{2\pi is_d/n})}.
\]
We shall see at the end of the proof that the left hand side of~\eqref{eqn:lemma_error_term_to_show} equals $nS_n$ for a particular choice of the polyhedron $P$.
\par
The $L^1$ norm of $F_n$ is defined to be
\[
\norm{F_n}_1=\frac{1}{(2\pi)^d}\int_0^{2\pi}\cdots\int_0^{2\pi}\bigabs{F_n(e^{i\theta_1},\dots,e^{i\theta_d})}\,d\theta_1\cdots d\theta_d.
\]
It is known (see~\cite[9.2.1]{TriBel2004}, for example) that
\begin{equation}
\norm{F_n}_1\le \gamma(P)(\log n)^d,   \label{eqn:L1_bound}
\end{equation}
where $\gamma(P)$ depends only on the polyhedron $P$. We shall find an upper bound for $S_n$ in terms of $\norm{F_n}_1$.
\par
Let $f(z)$ be a polynomial in $\C[z]$. By the mean value theorem there exist real numbers $\theta_0,\dots,\theta_{n-1}$ with $\theta_s\in[2\pi s/n,2\pi (s+1)/n]$ for all $s$ such that
\begin{equation}
\norm{f}_1=\frac{1}{2\pi}\sum_{s=0}^{n-1}\int_{2\pi s/n}^{2\pi(s+1)/n}\bigabs{f(e^{i\theta})}\,d\theta=\frac{1}{n}\sum_{s=0}^{n-1}\bigabs{f(e^{i\theta_s})}.   \label{eqn:L1_mean_value_thm}
\end{equation}
By the triangle inequality we have
\begin{align}
\left|\sum_{s=0}^{n-1}\bigabs{f(e^{i\theta_s})}-\sum_{s=0}^{n-1}\bigabs{f(e^{2\pi is/n})}\right|&\le \sum_{s=0}^{n-1}\bigabs{f(e^{i\theta_s})-f(e^{2\pi is/n})}   \nonumber\\
&=\sum_{s=0}^{n-1}\left|\int_{2\pi s/n}^{\theta_s}f'(e^{i\theta})\,d\theta\right|   \nonumber\\
&\le \int_{0}^{2\pi}\bigabs{f'(e^{i\theta})}\,d\theta   \nonumber\\
&=2\pi\norm{f'}_1.   \label{eqn:bound_L1_derivative}
\end{align}
Now suppose that $f$ has degree at most $n-1$. Then $\norm{f'}_1\le (n-1)\,\norm{f}_1$ by a Bernstein-type inequality (see~\cite[p.~143]{Bor2002} or~\cite[p.~11]{Zyg2002}, for example). Combination of~\eqref{eqn:L1_mean_value_thm} and~\eqref{eqn:bound_L1_derivative} then gives
\[
\sum_{s=0}^{n-1}\bigabs{f(e^{2\pi is/n})}\le (1+2\pi)n\,\norm{f}_1.
\]
Since $F_n(z_1,\dots,z_d)$ has degree at most $n-1$ in each indeterminate, we find by a straightforward induction that
\[
S_n\le (1+2\pi)^dn^d\,\norm{F_n}_1,
\]
and then with~\eqref{eqn:L1_bound},
\begin{equation}
S_n\le (1+2\pi)^d\gamma(P)(n\log n)^d.   \label{eqn:S_n_inequality}
\end{equation}
\par
Now we take $d=2q-1$ and
\[
P=\left\{(x_1,\dots,x_{2q-1})\in\R^{2q-1}:\stack{0\le x_1,\dots,x_{2q-1}\le 1,}{q-1\le x_1+\cdots+x_{2q-1}\le q}\right\}.
\]
Set $j_{2q}=q(n-1)-j_1-\dots-j_{2q-1}$ and $s_i=t_i-t_{2q}$ for all $i\in\{1,2,\dots,2q-1\}$ in~\eqref{eqn:lemma_error_term_to_show} to see that the left hand side of~\eqref{eqn:lemma_error_term_to_show} equals
\[
\sum_{t_{2q}\in\Z/n\Z}S_n=nS_n,
\]
so that the desired inequality~\eqref{eqn:lemma_error_term_to_show} follows from~\eqref{eqn:S_n_inequality}.
\end{proof}


\section{Fekete polynomials}
\label{sec:fekete}

In this section we prove Theorem~\ref{thm:norm_shifted_fekete} (and therefore also Theorem~\ref{thm:norm_fekete}) and Corollary~\ref{cor:norm_fekete}.
\par
We say that a tuple $(t_1,t_2,\dots,t_{2q})$ is \emph{even} if there exists a permutation~$\sigma$ of $\{1,2,\dots,2q\}$ such that $t_{\sigma(2k-1)}=t_{\sigma(2k)}$ for all $k\in\{1,2,\dots,q\}$. For example, $(2,1,1,3,2,3)$ is even, whereas $(2,1,1,3,1,3)$ is not even. Let $\Ec_q(n)$ be the set of even tuples in $(\Z/n\Z)^{2q}$.
\par
We begin with the following lemma.
\begin{lemma}
\label{lem:norm_even_tuples}
Let $q$ be a positive integer and let $f_p^r(z)$ be a shifted Fekete polynomial corresponding to the Fekete polynomial of degree $p-1$.~Then
\[
\lim_{p\to\infty}\left(\frac{\norm{f_p^r}_{2q}}{\sqrt{p}}\right)^{2q}=\lim_{p\to\infty}\frac{1}{p^{2q}}\sum_{t\in\Ec_q(p)}h_{p,r}(t),
\]
provided that one of the limits exists.
\end{lemma}
\begin{proof}
Let $f_p(z)$ be the Fekete polynomial of degree $p-1$. For $t\in(\Z/p\Z)^{2q}$, let $J_p(t)$ be the indicator function that equals one if $t$ is even and is zero otherwise. From Proposition~\ref{pro:norm_from_L_and_h} we find that
\[
\left(\frac{\norm{f_p^r}_{2q}}{\sqrt{p}}\right)^{2q}\!\!=\frac{1}{p^{2q}}\sum_{t\in(\Z/p\Z)^{2q}}\!\!J_p(t)h_{p,r}(t)+\frac{1}{p^{2q}}\sum_{t\in(\Z/p\Z)^{2q}}\!\!\left(L_{f_p}(t)-J_p(t)\right)h_{p,r}(t).
\]
We show that the second sum on the right hand side tends to zero. This will prove the lemma since 
\[
\sum_{t\in(\Z/p\Z)^{2q}}J_p(t)h_{p,r}(t)=\sum_{t\in\Ec_q(p)}h_{p,r}(t).
\]
Notice that $f_p(e_p(k))$ is a quadratic Gauss sum, whose explicit evaluation is~\cite{BerEvaWil1998}
\[
f_p(e_p(k))=i^{(p-1)^2/4}p^{1/2}\,\leg{k}{p}.
\]
Therefore
\[
L_{f_p}(t_1,\dots,t_{2q})=\frac{1}{p}\sum_{m=0}^{p-1}\leg{m+t_1}{p}\cdots\leg{m+t_{2q}}{p}.
\]
If $(t_1,\dots,t_{2q})$ is even, then it is readily verified that
\[
1-q/p\le L_{f_p}(t_1,\dots,t_{2q})\le 1-1/p.
\]
On the other hand, if $(t_1,\dots,t_{2q})$ is not even, then the Weil bound for sums over multiplicative characters~\cite[Lemma~9.25]{MonVau1997},~\cite[Theorem~5.41]{LidNie1997} gives
\[
\abs{L_{f_p}(t_1,\dots,t_{2q})}\le (2q-1)p^{-1/2}.
\]
Therefore
\[
\bigabs{L_{f_p}(t)-J_p(t)}\le (2q-1)p^{-1/2}\quad\text{for all $t\in(\Z/p\Z)^{2q}$}.
\]
By the triangle inequality we then find that
\[
\frac{1}{p^{2q}}\Biggabs{\sum_{t\in(\Z/p\Z)^{2q}}\left(L_{f_p}(t)-J_p(t)\right)h_{p,r}(t)}\le \frac{2q-1}{p^{2q+1/2}}\sum_{t\in(\Z/p\Z)^{2q}}\abs{h_{p,r}(t)},
\]
which tends to zero as $p\to\infty$ by Lemma~\ref{lem:error_term}, as required.
\end{proof}
\par
In what follows, we shall evaluate the right hand side of the expression in Lemma~\ref{lem:norm_even_tuples}.
\par
Let $t=(t_1,t_2,\dots,t_m)$ be a tuple in $(\Z/n\Z)^m$ and let $\pi\in\Pi_m$. We define $t\prec\pi$ to be true if and only if $t_j=t_k$ whenever $j$ and $k$ belong to the same block of $\pi$. For example, if $t=(1,2,1)$ and $\pi=\{\{1,3\},\{2\}\}$, then $t\prec\pi$ holds.
\par
\begin{lemma}
\label{lem:even_tuple_in_ex}
Let $h:\Ec_q(n)\to\C$ be an arbitrary function and let $T(k)$ be the $k$-th signed tangent number. Then
\begin{equation}
\sum_{t\in\Ec_q(n)}h(t)=\sums{\pi\in\Pi_{2q}\\[0.5ex]\text{$\pi$ even}}\;\sums{t\in\Ec_q(n)\\ t\prec \pi} h(t)\,\prod_{B\in\pi}T(\tfrac{1}{2}\abs{B}).   \label{eqn:in_ex_set_part}
\end{equation}
\end{lemma}
\par
To prove the lemma, we shall need the following combinatorial principle (see~\cite[p.~5]{Sta1999}, for example), in which $\N=\{1,2,3,\dots\}$.
\begin{lemma}
\label{lem:exp_formula}
Let $K$ be a field of characteristic $0$, let $f:\N\to K$ be arbitrary, and define a new function $g:\N\cup\{0\}\to K$ by $g(0)=1$ and
\[
g(k)=\sum_{\pi\in\Pi_k}\,\prod_{B\in\pi}f(\abs{B})\quad\text{for $k\ge 1$}.
\]
Let $G(z)=\sum_{k\ge 0}g(k)z^k/k!$ and $F(z)=\sum_{k\ge 1}f(k)z^k/k!$ be the corresponding exponential generating functions. Then $G(z)=\exp(F(z))$. Moreover,
\[
g(k)=\sum_{j=1}^k{k-1\choose j-1}f(j)g(k-j)\quad\text{for $k\ge 1$}.
\]
\end{lemma}
\begin{proof}
The first part of the lemma is a consequence of Fa\'a di Bruno's generalisation of the chain rule (see~\cite[Theorem~1.3.2]{KraPar2002}, for example), which states that, for a formal power series $E(z)$ and $k\ge 1$, we have
\[
(E\circ F)^{(k)}(z)=\sum_{\pi\in\Pi_k} (E^{(\abs{\pi})}\circ F)(z)\prod_{B\in\pi}F^{(\abs{B})}(z).
\]
Take $E(z)=\exp(z)$ and set $z=0$ to see that the right hand side equals $g(k)$, which proves the first part. The second part follows from $G'(z)=G(z)F'(z)$ by equating coefficients.
\end{proof}
\par
For a tuple $t\in(\Z/n\Z)^m$, let $\pi\in\Pi_m$ be the coarsest partition of $\{1,2,\dots,m\}$ with the property $t\prec \pi$ and define $m_k(t)$ to be the number of blocks $B$ in~$\pi$ such that $\abs{B}=k$. For example, if $t=(1,3,2,1,2)$, then the coarsest partition $\pi$ with $t\prec\pi$ is $\{\{1,4\},\{3,5\},\{2\}\}$ and we have $m_1(t)=1$, $m_2(t)=2$, and $m_k(t)=0$ for $k>2$.
\par
We now give a proof of Lemma~\ref{lem:even_tuple_in_ex}.
\begin{proof}[Proof of Lemma~\ref{lem:even_tuple_in_ex}]
Taking $F(z)=\log\cosh(z)$ in Lemma~\ref{lem:exp_formula} (so that $G(z)=\cosh(z)$), we find with~\eqref{eqn:log_cosh} and $\cosh(z)=\sum_{k\ge 0}z^{2k}/(2k)!$ that
\begin{equation}
\sums{\pi\in\Pi_{2k}\\[0.5ex]\text{$\pi$ even}}\;\prod_{B\in\pi}T(\tfrac{1}{2}\abs{B})=1\quad\text{for each $k\ge 1$}.   \label{eqn:sum_T}
\end{equation}
Let $s\in\Ec_q(n)$ be an even tuple. By linearity, it suffices to prove the lemma for the case that $h(x)=1$ for $x=s$ and $h(x)=0$ otherwise. Clearly, the left hand side of~\eqref{eqn:in_ex_set_part} equals $1$. On the other hand, the sum
\[
\sums{t\in\Ec_q(n)\\ t\prec \pi} h(t)
\]
is just the indicator function of the event $s\prec\pi$, so we can restrict the outer summation on the right hand side of~\eqref{eqn:in_ex_set_part} to the even partitions that are refinements of the coarsest partition $\pi\in\Pi_{2q}$ with the property $s\prec\pi$. Therefore the right hand side of~\eqref{eqn:in_ex_set_part} equals
\[
\prod_{k=1}^q\left( \sums{\pi\in\Pi_{2k}\\[0.5ex]\text{$\pi$ even}}\prod_{B\in\pi}T(\tfrac{1}{2}\abs{B})\right)^{m_k(s)},
\]
which again equals $1$ by~\eqref{eqn:sum_T}. 
\end{proof}
\par
Next we evaluate the inner sums in the right hand side of~\eqref{eqn:in_ex_set_part} for $h=h_{n,r}$.
\begin{lemma}
\label{lem:partial_sums_fekete}
Let $\pi=\{B_1,\dots,B_\ell\}\in\Pi_{2q}$ be an even partition with $\ell$ blocks. Write $N_i=\abs{B_i}/2$ and $P_i=\abs{\{x\in B_i:x>q\}}$. If $r/n\to R$ as $n\to\infty$,~then  
\[
\lim_{n\to\infty}\frac{1}{n^{2q}}\sums{t\in\Ec_q(n)\\t\prec\pi}h_{n,r}(t)=\!\sums{a_1,\dots,a_\ell\in\Z\\a_1+\cdots+a_\ell=q}\;\prod_{i=1}^\ell\;\frac{1}{(2N_i-1)!}{2N_i-1\bangle 2R(N_i-P_i)+a_i-1}.
\]
\end{lemma}
\par
To prove the lemma, we use the following asymptotic counting result, which follows from known results on the number of restricted integer compositions~\cite{FieAlf1991},~\cite{Ege2013} or, alternatively, from integration results over a simplex~\cite{GooTid1978}. By $I[E]$ we denote the indicator function of an event $E$.
\begin{lemma}
\label{lem:asymptotic_counting_lemma}
Let $N$ be a positive integer and let $M$ be real. Let $(m_n)$ be a sequence of integers such that $m_n/n\to M$ as $n\to\infty$. Then
\[
\lim_{n\to\infty}\frac{1}{n^{N-1}}\sum_{0\le j_1,\dots,j_N<n}I\big[j_1+\dots+j_N=m_n\big]=\frac{1}{(N-1)!}{N-1\bangle M-1}.
\]
\end{lemma}
\begin{proof}
It is well known (see~\cite[(11)]{FieAlf1991} or~\cite[Example~33]{Ege2013}, for example) that
\[
\sum_{0\le j_1,\dots,j_N<n}I\big[j_1+\dots+j_N=m_n\big]=\sum_{j=0}^N(-1)^j{N\choose j}{N+m_n-nj-1\choose N-1}.
\]
Since
\[
\lim_{n\to\infty}\frac{1}{n^{N-1}}{N+m_n-nj-1\choose N-1}=\frac{1}{(N-1)!}\,(\max(0,M-j))^{N-1},
\]
the lemma follows from the definition~\eqref{eqn:def:Eulerian_numbers} of the generalised Eulerian numbers.
\end{proof}
We now prove Lemma~\ref{lem:partial_sums_fekete}.
\begin{proof}[Proof of Lemma~\ref{lem:partial_sums_fekete}]
Put
\[
H_n=\sums{t\in\Ec_q(n)\\t\prec\pi}h_{n,r}(t).
\]
Let $\e_k=-1$ for $k\le q$ and $\e_k=1$ for $k>q$. Since
\[
h_{n,r}(t_1,\dots,t_{2q})=\sums{0\le j_1,\dots,j_{2q}<n\\j_1+\dots+j_q=j_{q+1}+\dots+j_{2q}}\prod_{i=1}^{\ell}\prod_{k\in B_i}e_n(\e_kt_k(j_k+r)),
\]
we can rewrite $H_n$ as
\[
H_n=\sums{0\le j_1,\dots,j_{2q}<n\\j_1+\dots+j_q=j_{q+1}+\dots+j_{2q}}\prod_{i=1}^\ell\;\sum_{t\in\Z/n\Z}\;e_n\bigg(t\sum_{k\in B_i}\e_k(j_k+r)\bigg).
\]
The product is either zero or equals $n^\ell$ and is nonzero exactly when there exist $a_1,\dots,a_\ell\in\Z$ such that
\begin{equation}
\sum_{k\in B_i}\e_k(j_k+r)=a_in   \label{eqn:sum_B_i_condition}
\end{equation}
for all $i\in\{1,\dots,\ell\}$. Hence
\[
H_n=n^\ell\sums{0\le j_1,\dots,j_{2q}<n\\j_1+\dots+j_q=j_{q+1}+\dots+j_{2q}}\sum_{a_1,\dots,a_\ell\in\Z}\;\prod_{i=1}^\ell\;I\Bigg[\sum_{k\in B_i}\e_k(j_k+r)=a_in\Bigg].
\]
Summing both sides of~\eqref{eqn:sum_B_i_condition} over $i\in\{1,\dots,\ell\}$ gives
\[
\sum_{k=1}^q(j_{q+k}-j_k)=n\sum_{i=1}^\ell a_i,
\]
so that
\[
H_n=n^\ell\sums{a_1,\dots,a_\ell\in\Z\\a_1+\cdots+a_\ell=0}\sum_{0\le j_1,\dots,j_{2q}<n}\;\prod_{i=1}^\ell\;I\Bigg[\sum_{k\in B_i}\e_k(j_k+r)=a_in\Bigg].
\]
The $i$-th factor within the inner sum depends only on $\abs{B_i}=2N_i$ of the summation variables in the inner sum, so that we can factor the inner sum as follows
\[
\prod_{i=1}^\ell\;\sum_{0\le j_1,\dots,j_{2N_i}<n}\;I\Bigg[\sum_{k=1}^{P_i}(j_k+r)-\sum_{k=P_i+1}^{2N_i}(j_k+r)=a_in\Bigg].
\]
Replace $j_k$ by $n-1-j_k$ for $k\in\{P_i+1,\dots,2N_i\}$ to see that this expression equals
\[
\prod_{i=1}^\ell\;\sum_{0\le j_1,\dots,j_{2N_i}<n}\;I\Bigg[\sum_{k=1}^{2N_i}j_k=(2N_i-P_i)(n-1)+2r(N_i-P_i)+a_in\Bigg].
\]
Since $\sum_{i=1}^\ell(2N_i-1)=2q-\ell$, we find from Lemma~\ref{lem:asymptotic_counting_lemma} that
\[
\lim_{n\to\infty}\frac{H_n}{n^{2q}}=\sums{a_1,\dots,a_\ell\in\Z\\a_1+\cdots+a_\ell=0}\;\prod_{i=1}^\ell\;\frac{1}{(2N_i-1)!}{2N_i-1\bangle 2N_i-P_i+2R(N_i-P_i)+a_i-1},
\]
since the outer sum is locally finite. The lemma follows after re-indexing and using $\sum_{i=1}^\ell (2N_i-P_i)=q$.
\end{proof}
\par
Theorem~\ref{thm:norm_shifted_fekete} and therefore Theorem~\ref{thm:norm_fekete} now follows from Lemmas~\ref{lem:norm_even_tuples},~\ref{lem:even_tuple_in_ex}, and~\ref{lem:partial_sums_fekete}. It remains to show how to deduce Corollary~\ref{cor:norm_fekete} from Theorem~\ref{thm:norm_fekete}. To do so, write
\begin{equation}
A_N(x)=\sum_{a=1}^{2N-1}{2N-1\bangle a-1}\,x^a,   \label{eqn:def_Eulerian_polynomials}
\end{equation}
which is known (after dividing by $x$) as an \emph{Eulerian polynomial}.  Letting $N_1,\dots,N_\ell$ be positive integers such that $N_1+\cdots+N_\ell=k$, we have
\[
\prod_{i=1}^\ell A_{N_i}(x)=\sum_{m=\ell}^{2k-\ell} x^m\sums{a_1,\dots,a_\ell\in\Z\\a_1+\cdots+a_\ell=m}\;\prod_{i=1}^\ell{2N_i-1\bangle a_i-1}.
\]
Define polynomials $F_k(x)$ by $F_k(x)=0$ for odd $k$, $F_0(x)=1$, and
\begin{equation}
F_{2k}(x)=\sums{\pi\in\Pi_{2k}\\[0.5ex]\text{$\pi$ even}}\,\prod_{i=1}^\ell \frac{T(N_i)\,A_{N_i}(x)}{(2N_i-1)!}\quad\text{for $k\ge 1$},   \label{eqn:def_F_2n}
\end{equation}
where $\pi=\{B_1,\dots,B_\ell\}$ and $N_i=\abs{B_i}/2$. Then $F_{2k}(x)$ is a polynomial of degree $2k-1$ with $F_{2k}(0)=0$ for $k\ge 1$, so we can write
\[
F_{2k}(x)=\sum_{m=1}^{2k-1} F(k,m)\,x^m\quad\text{for $k\ge 1$}.
\]
It is readily verified that Theorem~\ref{thm:norm_fekete} is equivalent to
\[
\lim_{p\to\infty}\left(\frac{\norm{f_p}_{2q}}{\sqrt{p}}\right)^{2q}=F(q,q).
\]
It remains to show that the numbers $F(k,m)$ are the same as those given in Corollary~\ref{cor:norm_fekete}. Use $F_0(x)=1$ and apply Lemma~\ref{lem:exp_formula} to~\eqref{eqn:def_F_2n} to find that
\[
F_{2k}(x)=\sum_{j=1}^k{2k-1\choose 2j-1}\frac{T(j)A_j(x)}{(2j-1)!}\,F_{2k-2j}(x)\quad\text{for $k\ge 1$}.
\]
With $F(0,0)=1$ (which equals $F_0(x)$), this is equivalent to the recursive definition of the numbers $F(k,m)$ given in Corollary~\ref{cor:norm_fekete}.


\section{Galois polynomials}
\label{sec:galois}

In this section we prove Theorem~\ref{thm:norm_galois} and Corollary~\ref{cor:norm_galois}. We use the following notation throughout this section. A tuple $(t_1,t_2,\dots,t_{2q})$ is an \emph{abelian square} if there exists a permutation~$\sigma$ of $\{1,2,\dots,q\}$ such that $t_{\sigma(k)}=t_{q+k}$ for all $k\in\{1,2,\dots,q\}$, so that the second half of the tuple is a permutation of the first half. Let $\Ac_q(n)$ be the set of abelian squares in $(\Z/n\Z)^{2q}$.
\begin{lemma}
\label{lem:norm_abelian_squares}
Let $q$ be a positive integer and let $g_n(z)$ be a Galois polynomial of degree $n-1$. Then
\[
\lim_{n\to\infty}\left(\frac{\norm{g_n}_{2q}}{\sqrt{n}}\right)^{2q}=\lim_{n\to\infty}\;\frac{1}{n^{2q}}\sum_{t\in\Ac_q(n)}h_{n,0}(t),
\]
provided that one of the limits exists.
\end{lemma}
\begin{proof}
For $t\in(\Z/n\Z)^{2q}$, let $J_n(t)$ be the indicator function that equals one if $t$ is an abelian square and is zero otherwise. From Proposition~\ref{pro:norm_from_L_and_h} we find that
\[
\left(\frac{\norm{g_n}_{2q}}{\sqrt{n}}\right)^{2q}\!\!=\frac{1}{n^{2q}}\!\!\sum_{t\in(\Z/n\Z)^{2q}}\!\!J_n(t)h_{n,0}(t)+\frac{1}{n^{2q}}\!\!\sum_{t\in(\Z/n\Z)^{2q}}\!\!\left(L_{g_n}(t)-J_n(t)\right)h_{n,0}(t).
\]
We show that the second expression on the right hand side tends to zero, which will prove the lemma. Write $s=n+1$, so that $s$ is a power of two. By definition, a Galois polynomial of degree $n-1$ can be written as 
\[
g_n(z)=\sum_{j=0}^{n-1}\psi(\theta^j)z^j,
\]
where $\psi$ is an additive character of $\F_s$ and $\theta$ is a primitive element of $\F_s$. For a multiplicative character $\xi$ of $\F_s$, we define the Gauss sum
\[
G(\xi)=\sum_{x\in\F_s^*}\psi(x)\xi(x).
\]
Letting $\chi$ be the multiplicative character of $\F_s$ given by $\chi(\theta)=e_n(1)$, we see that $g_n(e_n(k))=G(\chi^k)$ for all $k\in\Z/n\Z$.
Therefore
\[
L_{g_n}(t_1,\dots,t_{2q})=\frac{1}{n^{q+1}}\sum_{m\in\Z/n\Z}\prod_{k=1}^qG(\chi^{m+t_k})\overline{G(\chi^{m+t_{q+k}})}.
\]
Since $\abs{G(\xi)}^2$ equals $1$ if $\xi$ is trivial and equals $n+1$ otherwise, we find that $\abs{L_{g_n}(t_1,\dots,t_{2q})-1}=O(n^{-1})$ if $(t_1,\dots,t_{2q})$ is an abelian square. On the other hand, if $(t_1,\dots,t_{2q})$ is not an abelian square, then a result due to Katz~\cite[pp.~161--162]{Kat1988} shows that
\[
\abs{L_{g_n}(t_1,\dots,t_{2q})}\le \frac{q}{n^{q+1}}\,(n+1)^{q+1/2}.
\]
Therefore, by the triangle inequality,
\[
\frac{1}{n^{2q}}\Biggabs{\sum_{t\in(\Z/n\Z)^{2q}}\left(L_{g_n}(t)-J_n(t)\right)h_{n,0}(t)}=O(n^{-2q-1/2})\sum_{t\in(\Z/n\Z)^{2q}}\abs{h_{n,0}(t)},
\]
which tends to zero as $n\to\infty$ by Lemma~\ref{lem:error_term}, as required.
\end{proof}
\par
We proceed similarly as for Fekete polynomials and seek an asymptotic evaluation of the right hand side of the expression in Lemma~\ref{lem:norm_abelian_squares}.
\par
The following lemma is an analogue of Lemma~\ref{lem:even_tuple_in_ex}.
\begin{lemma}
\label{lem:abelian_square_in_ex}
Let $h:\Ac_q(n)\to\C$ be a function that depends only on the first~$q$ entries of its input and let $C(k)$ be the $k$-th signed Carlitz number. Then
\begin{equation}
\sum_{t\in\Ac_q(n)}h(t)=q!\sum_{\pi\in\Pi_q}\;\sums{u\in(\Z/n\Z)^q\\ u\prec \pi}\,h(u|u)\;\prod_{B\in\pi}\frac{C(\abs{B})}{\abs{B}!},   \label{eqn:in_ex_galois}
\end{equation}
where $u|u$ is the $(2q)$-tuple with the first and the second half equal to $u$.
\end{lemma}
\begin{proof}
Take $F(z)=\log J_0(2\sqrt{z})$ in Lemma~\ref{lem:exp_formula}, so that $G(z)$ equals
\[
J_0(2\sqrt{z})=\sum_{k=0}^\infty\frac{(-1)^k}{(k!)^2}z^k.
\]
Use~\eqref{eqn:def_carlitz_numbers} to find from Lemma~\ref{lem:exp_formula} that
\[
\sum_{\pi\in\Pi_k}\;\prod_{B\in\pi}\frac{(-1)^{\abs{B}}C(\abs{B})}{\abs{B}!}=\frac{(-1)^k}{k!}\quad\text{for each $k\ge 1$},
\]
or equivalently
\begin{equation}
\sum_{\pi\in\Pi_k}\;\prod_{B\in\pi}\frac{C(\abs{B})}{\abs{B}!}=\frac{1}{k!}\quad\text{for each $k\ge 1$}.   \label{eqn:sum_C}
\end{equation}
\par
Now let $v\in(\Z/n\Z)^q$ and let $V$ be the set of abelian squares in $(\Z/n\Z)^{2q}$ whose first $q$ entries equal those of $v$. By linearity, it suffices to prove the lemma for the case that $h(x)=1$ for $x\in V$ and $h(x)=0$ otherwise. Then the left hand side of~\eqref{eqn:in_ex_galois} equals
\begin{equation}
\abs{V}=\frac{q!}{\prod_{k=1}^q(k!)^{m_k(v)}}   \label{eqn:num_as_completion}
\end{equation}
(where $m_k(v)$ was defined before the proof of Lemma~\ref{lem:even_tuple_in_ex}). On the other hand, the right hand side of~\eqref{eqn:in_ex_galois} equals
\[
q!\prod_{k=1}^q\left(\sum_{\pi\in\Pi_k}\prod_{B\in\pi}\frac{C(\abs{B})}{\abs{B}!}\right)^{m_k(v)},
\]
which by~\eqref{eqn:sum_C} equals~\eqref{eqn:num_as_completion} again.
\end{proof}
\par
Next we evaluate the inner sums in the right hand side of~\eqref{eqn:in_ex_galois} for $h=h_{n,0}$.
\begin{lemma}
\label{lem:partial_sums_galois}
Let $\pi=\{B_1,\dots,B_\ell\}\in\Pi_q$ be a partition with $\ell$ blocks and write $N_i=\abs{B_i}$. Then  
\[
\lim_{n\to\infty}\frac{1}{n^{2q}}\sums{u\in(\Z/n\Z)^q\\u\prec\pi}h_{n,0}(u|u)=\sums{a_1,\dots,a_\ell\in\Z\\a_1+\cdots+a_\ell=q}\;\prod_{i=1}^\ell\;\frac{1}{(2N_i-1)!}{2N_i-1\bangle a_i-1},
\]
where $u|u$ is the $(2q)$-tuple with the first and the second half equal to $u$.
\end{lemma}
\begin{proof}
The proof is similar to that of Lemma~\ref{lem:partial_sums_fekete}, and so is presented in slightly less detail. Put
\[
H_n=\sums{u\in(\Z/n\Z)^q\\u\prec\pi}h_{n,0}(u|u),
\]
which we can rewrite as
\[
H_n=\sums{0\le j_1,\dots,j_{2q}<n\\j_1+\dots+j_q=j_{q+1}+\dots+j_{2q}}\prod_{i=1}^\ell\;\sum_{u\in \Z/n\Z}\;e_n\left(u\sum_{k\in B_i}(j_{q+k}-j_k)\right).
\]
The product is either zero or equals $n^\ell$ and is nonzero exactly when there exist $a_1,\dots,a_\ell\in\Z$ such that
\begin{equation}
\sum_{k\in B_i}(j_{q+k}-j_k)=a_in   \label{eqn:sum_B_i_condition_2}
\end{equation}
for all $i\in\{1,\dots,\ell\}$. Hence
\[
H_n=n^\ell\sums{0\le j_1,\dots,j_{2q}<n\\j_1+\dots+j_q=j_{q+1}+\dots+j_{2q}}\sum_{a_1,\dots,a_\ell\in\Z}\;\prod_{i=1}^\ell\;I\Bigg[\sum_{k\in B_i}(j_{q+k}-j_k)=a_in\Bigg].
\]
Summing both sides of~\eqref{eqn:sum_B_i_condition_2} over $i\in\{1,\dots,\ell\}$ gives
\[
\sum_{k=1}^q(j_{q+k}-j_k)=n\sum_{i=1}^\ell a_i,
\]
so that
\[
H_n=n^\ell\sums{a_1,\dots,a_\ell\in\Z\\a_1+\cdots+a_\ell=0}\sum_{0\le j_1,\dots,j_{2q}<n}\;\prod_{i=1}^\ell\;I\Bigg[\sum_{k\in B_i}(j_{q+k}-j_k)=a_in\Bigg]
\]
or equivalently
\[
H_n=n^\ell\sums{a_1,\dots,a_\ell\in\Z\\a_1+\cdots+a_\ell=0}\sum_{0\le j_1,\dots,j_{2q}<n}\;\prod_{i=1}^\ell\;I\Bigg[\sum_{k\in B_i}(j_{q+k}+j_k)=a_in+N_i(n-1)\Bigg].
\]
We can factor the inner sum as follows
\[
\prod_{i=1}^\ell\;\sum_{0\le j_1,\dots,j_{2N_i}<n}\;I\Bigg[\sum_{k=1}^{2N_i}j_k=a_in+N_i(n-1)\Bigg].
\]
Since $\sum_{i=1}^\ell(2N_i-1)=2q-\ell$, we find from Lemma~\ref{lem:asymptotic_counting_lemma} that
\[
\lim_{n\to\infty}\frac{H_n}{n^{2q}}=\sums{a_1,\dots,a_\ell\in\Z\\a_1+\cdots+a_\ell=0}\;\prod_{i=1}^\ell\;\frac{1}{(2N_i-1)!}{2N_i-1\bangle N_i+a_i-1},
\]
since the outer sum is locally finite. The lemma follows after re-indexing the summation.
\end{proof}
\par
Theorem~\ref{thm:norm_galois} now follows from Lemmas~\ref{lem:norm_abelian_squares},~\ref{lem:abelian_square_in_ex}, and~\ref{lem:partial_sums_galois}, upon noting that $h_{n,0}$ has the required property in Lemma~\ref{lem:abelian_square_in_ex}.
\par
Next we deduce Corollary~\ref{cor:norm_galois} from Theorem~\ref{thm:norm_galois}. This is again broadly similar to the proof of Corollary~\ref{cor:norm_fekete}. Recall the definition of the Eulerian polynomials $A_N(x)$ from~\eqref{eqn:def_Eulerian_polynomials} and define polynomials $G_k(x)$ by $G_0(x)=1$, and
\begin{equation}
\frac{G_k(x)}{k!}=\sums{\pi\in\Pi_k}\,\prod_{i=1}^\ell \frac{C(N_i)\,A_{N_i}(x)}{(2N_i-1)!N_i!}\quad\text{for $k\ge 1$},   \label{eqn:def_G_n}
\end{equation}
where $\pi=\{B_1,\dots,B_\ell\}$ and $N_i=\abs{B_i}$. Then $G_k(x)$ is a polynomial of degree $2k-1$ with $G_k(0)=0$ for $k\ge 1$, so we can write
\[
G_k(x)=\sum_{m=1}^{2k-1}G(k,m)\,x^m\quad\text{for $k\ge 1$}.
\]
It is readily verified that Theorem~\ref{thm:norm_galois} is equivalent to
\[
\lim_{n\to\infty}\left(\frac{\norm{g_n}_{2q}}{\sqrt{n}}\right)^{2q}=G(q,q).
\]
It remains to show that the numbers $G(k,m)$ are the same as those given in Corollary~\ref{cor:norm_galois}. Use $G_0(x)=1$ and apply Lemma~\ref{lem:exp_formula} to~\eqref{eqn:def_G_n} to find that
\[
\frac{G_k(x)}{k!}=\sum_{j=1}^k{k-1\choose j-1}\frac{C(j)A_j(x)}{(2j-1)!\,j!}\,\frac{G_{k-j}(x)}{(k-j)!}\quad\text{for $k\ge 1$},
\]
or equivalently
\[
G_k(x)=\sum_{j=1}^k{k\choose j}{k-1\choose j-1}\frac{C(j)A_j(x)}{(2j-1)!}\,G_{k-j}(x)\quad\text{for $k\ge 1$}.
\]
With $G(0,0)=1$ (which equals $G_0(x)$), this is equivalent to the recursive definition of the numbers $G(k,m)$ given in Corollary~\ref{cor:norm_galois}.

%

\enlargethispage{3.5ex}

\end{document}